\documentclass[reqno, 11pt]{amsart}

\usepackage{amssymb,amsmath,amsthm,amsfonts}

\usepackage{graphicx}

\usepackage{afterpage}

\graphicspath{{.}}

\newtheorem{Theorem}{Theorem}[section]
\newtheorem{Lemma}[Theorem]{Lemma}
\newtheorem{Proposition}[Theorem]{Proposition}
\newtheorem{Corollary}[Theorem]{Corollary}

\theoremstyle{remark}
\newtheorem{Remark}[Theorem]{Remark} 

\newtheorem{Definition}[Theorem]{Definition}

\begin{document}

\def\Z{{\mathcal Z}}
\def\V{{\mathcal V}}
\def\F{{\mathcal F}}
\def\I{{\mathcal I}}
\def\U{{\mathcal U}}
\def\Iz{{\mathcal I_0}}
\def\H{{\mathcal H}}
\def\G{{\mathcal G}}
\def\Gp{{\mathcal {G'}}}
\def\Hp{{\mathcal {H'}}}
\def\R{{\mathcal R}}
\def\E{{\mathcal E}}

\def\Paths{{\mathcal P}}

\def\variety{Y}
\def\genvariety{\Z^{m,n}_{r,k}}
\def\basevar{\Z^{m,n}_2}

\def\LT{LT(J)}
\def\SC{\Delta_{\LT}}

\def\A{{\mathbf {A}}}

\def\del#1#2#3#4{\delta_{[#1,#2][#3,#4]}}
\def\eps#1#2#3#4{\epsilon_{[#1,#2][#3,#4]}}
\def\pho#1#2#3#4#5#6{\rho_{[#1,#2,#3][#4,#5,#6]}}
\def\lam#1#2#3#4#5#6{\lambda_{[#1,#2,#3][#4,#5,#6]}}
\def\sigh#1#2#3#4#5#6{\psi_{[#1,#2,#3][#4,#5,#6]}}

\def\mp{\marginpar}

\hfill To appear in \textit{Journal of Pure and Applied Algebra} \\


\title[Complex Associated to Jet Scheme of Determinantal Variety]{Initial Complex Associated to a Jet Scheme of a Determinantal Variety}
\author{Boyan Jonov}
\address{Dept. of Mathematics\\California State University
Northridge\\Northridge CA 91330\\U.S.A.\\Current address: Dept of Mathematics\\University of California Santa Barbara\\Santa Barbara CA 93106} \email{boyan@math.ucsb.edu}




\begin{abstract}
We show in this paper that the principal component of the first order jet scheme over
the classical determinantal variety of $m\times n$ matrices of rank at most
$1$ is arithmetically Cohen-Macaulay, by showing that an associated Stanley-Reisner simplicial complex is shellable.
\end{abstract}

\maketitle


\section{Introduction} \label{SecnIntro}

Let $F$ be an algebraically closed field and $\A_F^{k}$ the affine
space of dimension $k$ over $F$. By a \textsl{variety} in
$\A_F^{k}$ we will mean the zero set of a collection of
polynomials over $F$ in $k$ variables; in particular, our
varieties are not assumed irreducible. In \cite{KoSe1} and \cite{KoSe2}, Ko\v{s}ir
and Sethuraman had studied jet schemes over classical determinantal varieties, and
had described their components in a large number of cases. In particular, they had
shown  that the variety of first-order jets, or loosely the ``algebraic tangent
bundle,'' over the determinantal variety of $m\times n$ matrices ($m \le n$) of rank
at most
$1$ has two components when $m \ge 3$.  One component is simply the affine space
$\A_F^{mn}$ supported over the origin.  The other component, which is much more
interesting, is the closure of the set of
tangents at the nonsingular points of the base determinantal variety.  We denote
this component by $Y$, and refer to it as the \textsl{principal component.} (When
$m=2$, the variety of first-order jets is irreducible, and coincides with the
principal component $Y$.) The goal of this paper is to show that $Y$ is
arithmetically Cohen-Macaulay, i.e., its coordinate ring is Cohen-Macaulay.

Consider the truncated polynomial ring
$F[t]/(t^2)$, and let $X(t)=(f_{i,j}(t))_{i,j}$ be the generic $m\times n$ ($m \le n$)
matrix over this ring; thus, the $(i,j)$  entry of $X(t)$ is of the
form $ f_{i,j}(t) = x_{i,j} + y_{i,j} t $, where $1\le i \le m$, $1 \le j
\le n$, and $ x_{i,j}, y_{i,j}$ are variables.
Let $I$ be the ideal of $R = F[x_{i,j}, y_{i,j}]$, $1 \leq i \leq m, 1 \leq
j \leq n$, generated by the coefficients of powers of $t$ in each $2
\times 2$ minor of the generic matrix $X(t)$. Then the variety of first-order jets
over the $m\times n$ matrices ($m \le n$) of rank at most
$1$ is precisely the zero set of $I$.  Let $J$ be the ideal of the principal
component $Y$.  In \cite{KoSe2}, Ko\v{s}ir and Sethuraman showed that $I$ is
radical, and further, determined a Groebner basis for both $I$ and $J$ for the
graded reverse
lexicographical order using the following scheme: $ y_{1,1} >
y_{1,2}> \dots> y_{1,n} > y_{2,1} > \dots >y_{2,n}>\dots> y_{m,n}>
x_{1,1} > x_{1,2} > \dots> x_{1,n} > x_{2,1} > \dots >x_{2,n}>\dots>
x_{m,n} $ (see \cite[Theorem 2.4]{KoSe2},  \cite[Proposition 3.3]{KoSe2}, and also
\cite[Remark 2.2]{KoSe2}).

 It follows easily from the description in \cite[Theorem 2.4]{KoSe2} of the Groebner
basis $G$ of $J$ that the leading term ideal of $J$, $\LT := \langle lm(g); g\in G
\rangle$, is generated by the following family of monomials:

\begin{Proposition} \label{CorGBI0}
(Generators of $\LT$) The following families of monomials generate
$\LT$:  $A =\{x_{i,l}x_{j,k}\ |\ 1 \le i < j\le m,\ 1\le k < l\le
n\}$, $B = \{ x_{i,k}y_{j,l}\ |\ 1\le i < j \le m,\ 1\le k < l \le
n\}$, $C = \{x_{k,p}y_{j,q}y_{i,r}\ |\ 1 \le i < j \le k \le m,\ 1
\le p < q < r \le n\}$, $D = \{x_{i,r}y_{j,q}y_{k,p}\ |\ 1 \le i < j
< k \le m,\ 1 \le p < q \le r \le n\}$, and $E =
\{y_{i,r}y_{j,q}y_{k,p}\ |\ 1 \le i < j < k \le m,\ 1 \le p < q < r
\le n\}$.
\end{Proposition}


Since $\LT$ is generated by squarefree monomials
we can
construct the Stanley-Reisner complex $\SC$ of $\LT$: this is the
simplicial complex on vertices $\{x_{i,j},y_{i,j}\ |\ 1 \le i \le
m, 1 \le j \le n\}$ whose corresponding Stanley-Reisner ideal (see
\cite[Chap. 1]{MiSt} or \cite[Chap. 5]{BrHe}) is $\LT$.  The simplicial complex is
defined by the relation
 $x_{i_1,j_1}
\dots x_{i_k, j_k}
\dots x_{i_s, j_s}y_{i^{'}_1, j^{'}_1}
\dots y_{i^{'}_l, j^{'}_l}
\dots y_{i^{'}_t, j^{'}_t}$, $\ 1 \le i_k, i^{'}_l \le m, 1 \le j_k, j^{'}_l \le n$,
is a face of $\SC$ if, as a monomial, it does not belong to $\LT$.

We will enumerate all the facets of $\SC$ and we will describe an explicit ordering
of the facets which will show that $\SC$ is a shellable simplicial complex. By
standard results, shellability of $\SC$ allows us to conclude the following main
result of the paper (a result that has been independently obtained by Smith and Weyman  in \cite{SW} as well, using their geometric technique for computing syzygies):

\begin{Theorem} \label{Cohen-Macaulay}
The coordinate ring of $Y$, i.e., $R/J$, is Cohen-Macaulay.
\end{Theorem}

We wish to thank Professor Aldo Conca for some very valuable discussions
during the writing of the paper. We also wish to thank Professor Toma\v{z} Ko\v{s}ir  for being generous with his time and his encouragement. This paper constitutes our M.S. thesis
at California State University Northridge, and we wish to thank
Professor B.A. Sethuraman for suggesting this problem and for his
constant encouragement.



\section{Describing the Facets of $\SC$} \label{SecnElementsSC}

\begin{figure}[h]
\includegraphics[scale=0.35]{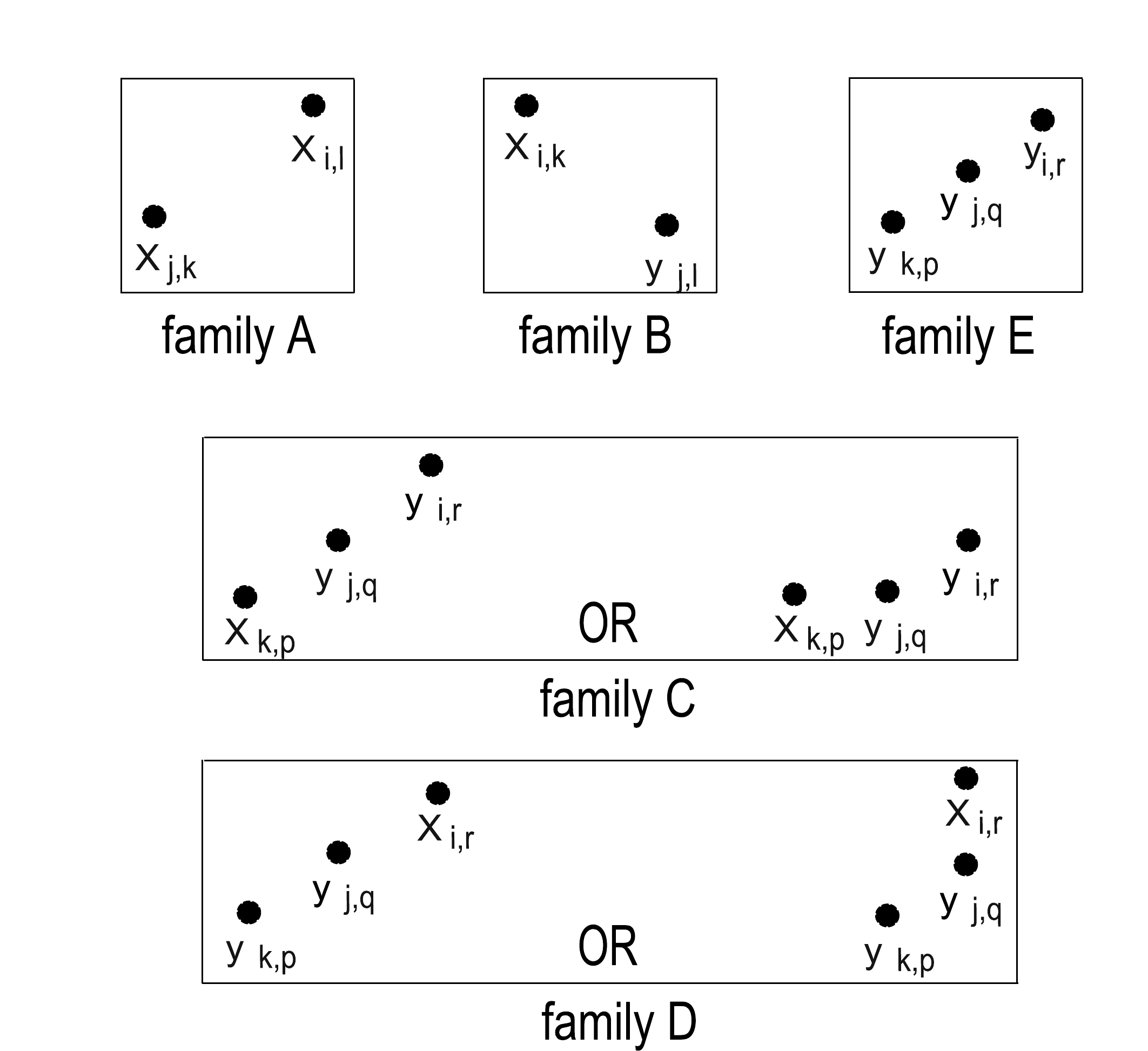}
\caption{ \label{family2}}
\end{figure}

It would be helpful in what follows to visualize the structure of
the monomials in the families $A$, $B$, $C$, $D$, and $E$ as described in Proposition \ref{CorGBI0}. For this, see Figure \ref{family2}. In this paper, we will visualize a monomial as being positioned in a matrix, where each variable of the monomial is
located in the matrix's entry corresponding to the index of the
variable.

In this section, we will enumerate all facets of $\SC$.  First, some
notation: we will denote a facet $F$ of the simplicial complex $\SC$
by $F=F_{x}F_{y}$, where $F_{x}$ is a string composed of vertices
$x_{i,j}$'s and $F_{y}$ is a string composed of vertices
$y_{i,j}$'s.  We will view each of $F_x$ and $F_y$ as both strings of vertices
or monomials, depending on the context. Note that $F_{x}F_{y} \in
\SC$ if, as a monomial, $F_{x}F_{y}$ does not belong in the ideal
$\LT$ if and only if $F_{x}F_{y}$ is not divisible by the generators of $\LT$ (see Proposition \ref{CorGBI0}).

We will start by showing a relation between the facets
of $\SC$ and those of the corresponding simplicial complexes
arising from classical determinantal varieties.   We refer to the
excellent survey paper of Bruns and Conca \cite{BrCo}. In this
paper,  the authors consider the facets of $\Delta_t$: the
Stanley-Reisner complex attached to the ideal $LT(I_t)$, which is
generated by the leading terms of the $t\times t$ minors of the
generic $m\times n$ matrix $(w_{i,j})$. The order they use is one
in which the leading term of a minor is the main diagonal, and it
is known that the leading terms of the $t\times t$ minors generate
the ideal of leading terms of $I_t$.

 The key result for us is \cite[Prop.
6.4]{BrCo}, where they enumerate the facets of $\Delta_t$. This is
a purely combinatorial result that enumerates the maximal subsets
of $V = \{ w_{i,j}: i \leq m, j \leq n \}$ that intersect any
$t$-subset of $V$ arising from the diagonal of some $t\times t$
submatrix of $(w_{i,j})$ in at most $t-1$ places, and can be
applied by symmetry to enumerate the maximal subsets of $V$ that
intersect any $t$-subset of $V$ arising from the antidiagonal of
some $t\times t$ submatrix of $(w_{i,j})$ in at most $t-1$ places.
We quote this result as:
\begin{Proposition} \label{Prop_classical_facets}(\cite[Prop.
6.4]{BrCo}).  Let $I_t$ be the ideal of $F[\{w_{i,j}\}]$ generated
the $t\times t$ minors of the generic $m\times n$ matrix
$(w_{i,j})$.  Write $LT(I_t)$ for the ideal generated by the lead
terms of the $t\times t$ minors  with respect to the graded
reverse lexicographical order $w_{1,1} > w_{1,2} > \dots> w_{1,n}
> w_{2,1} > \dots
>w_{2,n}>\dots> w_{m,n}$.  Write $\Delta_t$ for the Stanley-Reisner
complex of $LT(I_t)$.  Then the facets of $\Delta_t$ correspond to
all families of non-intersecting paths from
$w_{1,1},w_{2,1},\dots,w_{t-1,1}$ to $w_{m,n},w_ {m,n-1},\dots w_{m,n-t+2}$.
\end{Proposition}
Here, a path from $w_{a,b}$ to $w_{c,d}$, given $a\le c$ and $b\le d$,
is a sequence of vertices starting at $w_{a,b}$ and ending
at $w_{c,d}$ where each vertex in the sequence is either one step to
the right or one step down from the previous vertex.  A
non-intersecting path of the kind described in the last line of the
proposition above is a union of paths from $w_{i,1}$ to $w_{m,n-i+1}$
whose pairwise intersection is empty.  (It is known that for
the graded reverse lexicographic order as well, the leading terms
of the $t\times t$ minors generate the ideal of leading terms of
$I_t$.)

We observe that the monomials in $A$ correspond to the generators
of $LT(I_2)$  and the monomials in $E$ correspond to the
generators of $LT(I_3)$ (with the order specified in Proposition \ref{Prop_classical_facets}). So, for a facet facet $F=F_{x}F_{y}$ of $\SC$,  we have that $F_x$ is not in $LT(I_2)$ and $F_y$ is not in $LT(I_3)$. Therefore, by Proposition \ref{Prop_classical_facets}, we can state the following lemma:

\begin{Lemma} \label{lemma_Fx and Fy subsets of classical case}
$F_x$ is a subset of a path from $x_{1,1}$ to $x_{m,n}$ and $F_y$
is a subset of a pair of non-intersecting paths from $y_{1,1}$,
$y_{2,1}$  to $y_{m,n}$, $y_{m,n-1}$
\end{Lemma}

We will continue by showing that for each facet $F=F_{x}F_{y}$ of $\SC$,
$F_{x}$ is a non-empty string that contains at least two $x$-vertices. It is straightforward to see that $x_{m,n}F$  can not be divisible by any of the generators of  $\LT$. Hence, maximality of $F$ implies that $x_{m,n}$ is already in $F$. The next lemma shows that, in addition to $x_{m,n}$, $F$ must contain another $x$-vertex.

\begin{Lemma} \label{other x's present}
Let $F$ be a facet of the simplicial complex $\SC$. Then $F$ must contain at least two $x$-vertices, one of which is $x_{m,n}$ .
\end{Lemma}
\begin{proof}
We already know that $F$ must have $x_{m,n}$. Suppose that it is the only $x$-vertex that $F$ has. Consider then $x_{m-1,n}F$. We can easily check then that $x_{m-1,n}F$ is not divisible by any of the monomials in $A$, $B$, $C$, $D$, or $E$ . So, $x_{m-1,n}F \in \SC$ and maximality of $F$ implies that $x_{m-1,n}$ must already be in $F$, a contradiction to the assumption that the only $x$-vertex that $F$ contains is $x_{m,n}$.
\end{proof}


\noindent\textsl{Notation:} Let $F=F_{x}F_{y}$ be any facet and recall that, by Lemma \ref{lemma_Fx and Fy subsets of classical case}, $F_x$ is a subset of a path from $x_{1,1}$ to $x_{m,n}$. Thus, for any two $x$-vertices in $F_x$, one is always to the north, west, or north-west of the other. Let  $\mu(F)$ denote the $x$-vertex that is furthest to north and furthest to the west of all other $x$-vertices in $F_x$.  Thus, $\mu(F) = x_{i,j}$ implies that $i \le c$ and $j\le d$ for all $x_{c,d}$ in $F_x$ (see Figure \ref{facet}). Notice that $\mu(F) \neq x_{m,n}$ (Lemma \ref{other x's present}).

The next lemma deals with the $F_y$ part of a facet $F$ and, in particular, the lemma lists some of the $y$-vertices that must be present in a given facet.

\begin{Lemma} \label{y's present}
Let $F=F_{x}F_{y}$ be a facet of the simplicial complex $\SC$ with $\mu(F) = x_{i,j}$. Then $F$ must contain $y_{i,n}$ and $y_{m,j}$.
\end{Lemma}

\begin{proof}

To prove that $F$ contains $y_{i,n}$, it suffices to show that $y_{i,n}F$ is not divisible by a monomial in $A$, $B$, $C$, $D$, or $E$. Maximality of $F$  would then imply that $y_{i,n}$ must be in $F$.


Obviously, $y_{i,n}F$ can not be divisible by a monomial in $A$.
Also, $y_{i,n}F$ can not be divisible by a monomial in $B$ because otherwise it is easy to see that $y_{i,n}$ would have to be to the south-east of $\mu(F) = x_{i,j}$ - a contradiction. If $y_{i,n}F$ were divisible by a monomial in $D$, then another straightforward verification shows that $y_{i,n}$ must in a row below $\mu(F) = x_{i,j}$, which is impossible.




Suppose that $y_{i,n}F$ is divisible by a monomial in $C$. Then,
there must be some $x_{c,d}$ and $y_{s,t}$ in $F$ such that
$x_{c,d}y_{s,t}y_{i,n}$ is in $C$ (recall Figure \ref{family2}). But then the only possible location of $y_{s,t}$ is to the south-east of $x_{i,j}$. However, $x_{i,j}y_{s,t}$ is in $B$ and in $F$ - a contradiction.



Finally, suppose that $y_{i,n}F$ is divisible by a monomial in
$E$. Then there must be some $y_{a,b}$ and $y_{c,d}$ in $F$ such
that $y_{i,n}y_{a,b}y_{c,d}$ is in $E$ (recall Figure \ref{family2}).
In particular, it must be the case that, say, $y_{c,d}$ is to the
south-west of $y_{a,b}$ which, in turn, is to the south-west of
$y_{i,n}$. But then either $x_{i,j}y_{a,b}$ is in $B$ or $x_{i,j}y_{c,d}y_{a,b}$ is in $D$ - a contradiction in both cases.


So, $y_{i,n}F$ is not divisible by a monomial in $A$, $B$, $C$,
$D$, or $E$ which implies, as argued above, that $F$ must contain
$y_{i,n}$. We can similarly show that $F$ must contain $y_{m,j}$
as well.
\end{proof}

Now we are ready to describe the structure of all facets of $\SC$.
The following notation will be useful in the next theorem: for a
given facet $F=F_{x}F_{y}$ with $\mu(F) = x_{i,j}$, consider the
following partition of the $y$-vertices based on the index
$(i,j)$: $R_1 = \{y_{s,t} \ | \ s \leq i, j < t \}$, $R_2 =
\{y_{s,t} \ | \ s \leq i, t \leq j \}$, $R_3 = \{y_{s,t} \ | \ i <
s, t \leq j\}$, $R_4 = \{y_{s,t} \ | \ i < s, j < t\}$ (see Figure
\ref{regions1}) .

\begin{figure}[h]
\includegraphics[scale=0.15]{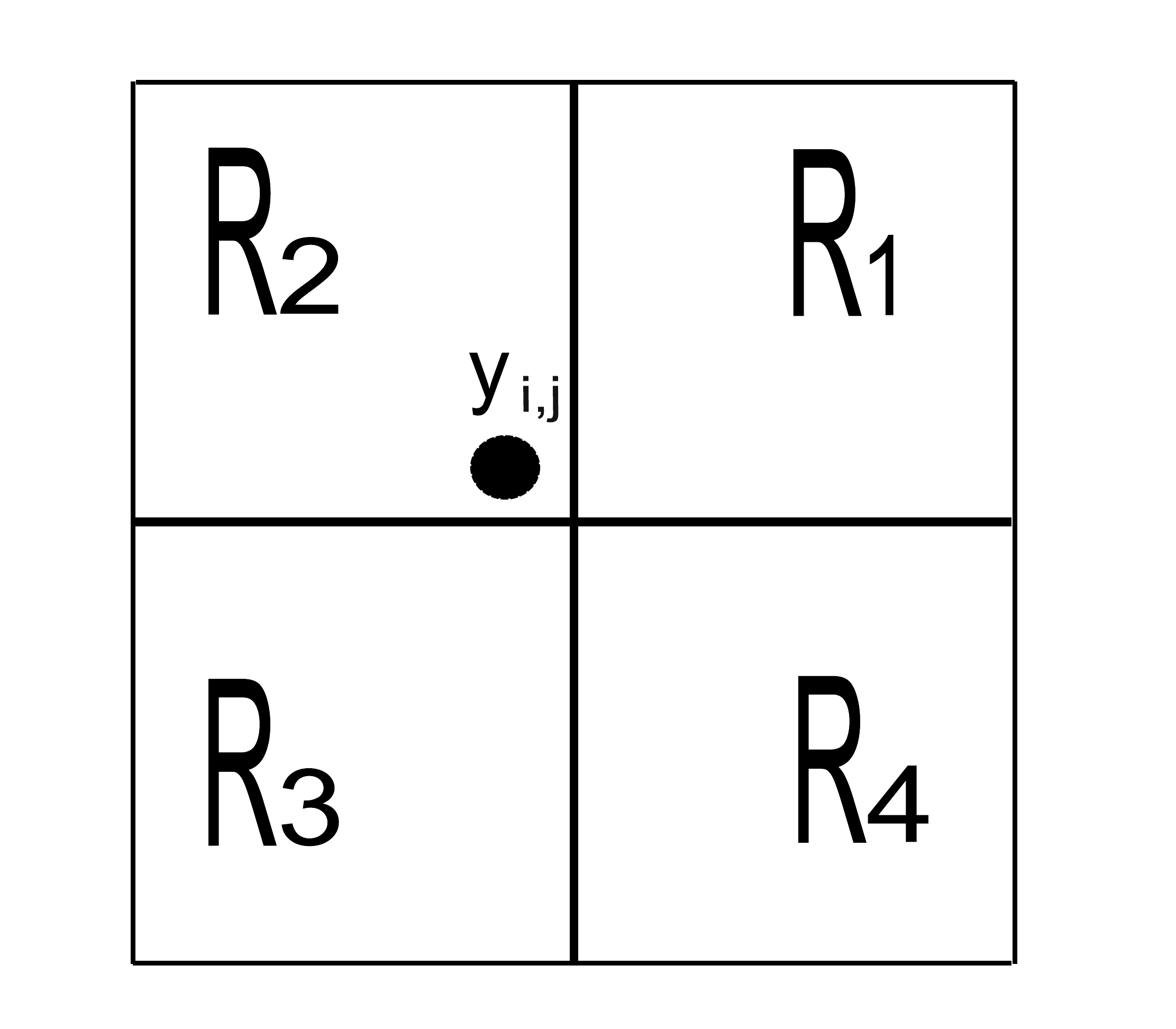}
\caption{ \label{regions1}}
\end{figure}

\begin{Theorem} \label{facets}
Let $F=F_{x}F_{y}$ be a facet of the simplicial complex $\SC$ with $\mu(F) = x_{i,j}$.
Then $F_{x}$ is a path from $x_{i,j}$ to  $x_{m,n}$ and
$F_{y}$ is a family of non-intersecting paths from $y_{1,1}$, $y_{2,1}$ to $y_{i,n}$, $y_{m,j}$.
\end{Theorem}

\begin{proof}


We will first show that $F=F_{x}F_{y}$ as described in the theorem is indeed a valid facet of $\SC$. Then we will argue that any facet of $\SC$ must have that form.


Let $F=F_{x}F_{y}$ with $F_{x}$  a path from $x_{i,j}$ to  $x_{m,n}$ and $F_{y}$  a family of non-intersecting paths from $y_{1,1}$, $y_{2,1}$ to $y_{i,n}$, $y_{m,j}$ be given (see Figure \ref{facet}). We will first show that $F$ is a facet of $\SC$ , i.e. $F$ is not divisible by monomials in $A$, $B$, $C$, $D$, or $E$ and $F$ is maximal with respect to inclusion.

\begin{figure}[h]
\includegraphics[scale=0.30]{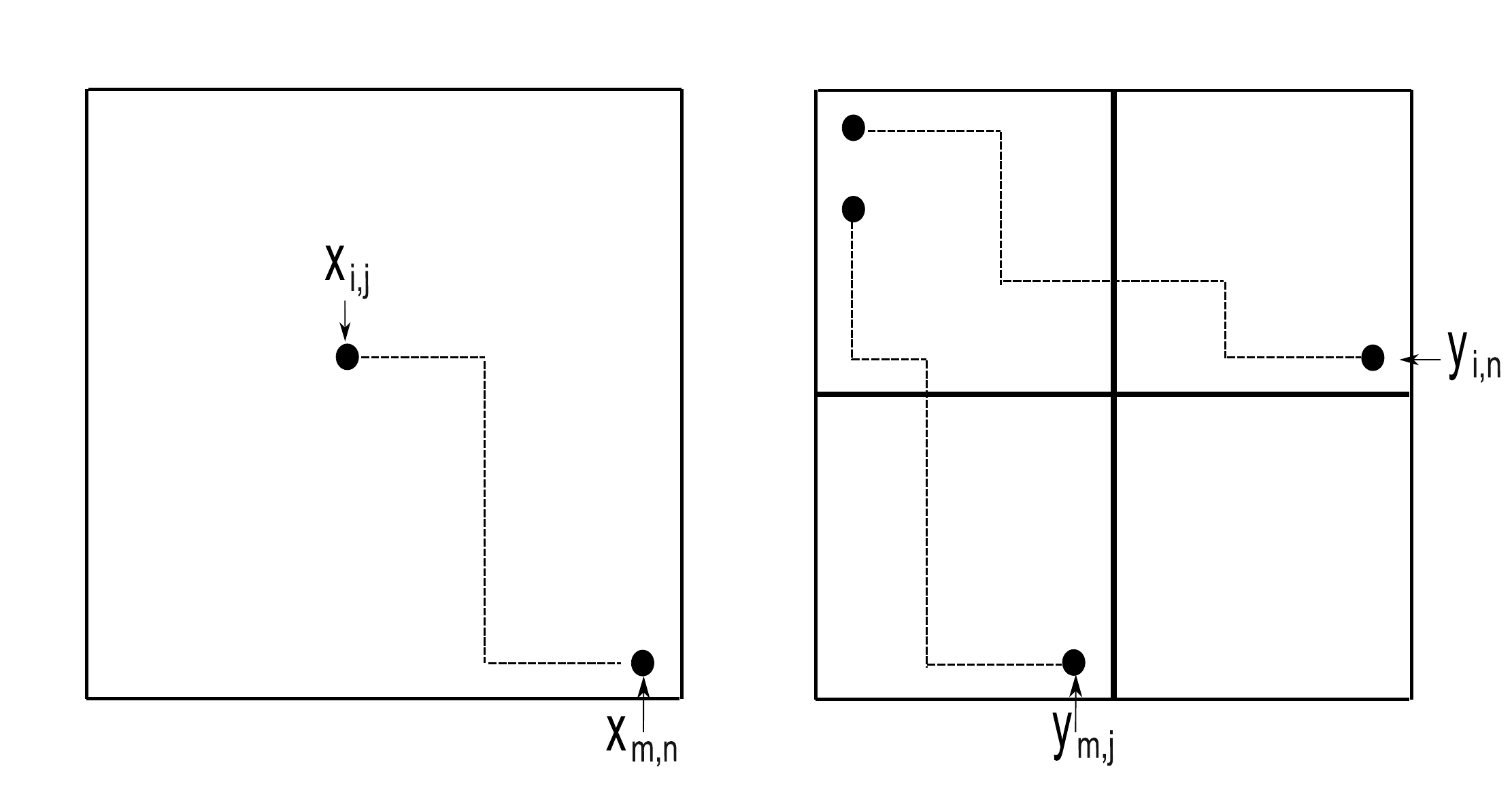}
\caption{ \label{facet}}
\end{figure}

Obviously, $F$ is not divisible by monomials in $A$ and $E$. To see that $F$ is not divisible by monomials in $B$, $C$, and $D$, it is enough to notice that F doesn't contain a $y$-variable in $R_4$ or variables of the form $y_{c,d}$ and $y_{e,f}$ such that one is to the south-west of the other and both are entirely in $R_1$ or $R_3$.



Next, we will show that $F=F_{x}F_{y}$ is maximal with respect to inclusion by arguing that any vertex attached to $F$ would make the resulting monomial divisible by some monomial in $A$, $B$, $C$, $D$, or $E$ (i.e. that resulting monomial can not be a face in $\SC$). Recall that by Lemma \ref{lemma_Fx and Fy subsets of classical case} $F_x$ is a subset of a path from $x_{1,1}$ to $x_{m,n}$. So, if we attach a vertex $x_{a,b}$ to $F_x$, it has to be to the north, west or north-west of $x_{i,j}$. But notice that in this case either $x_{a,b}y_{m,j}$ or $x_{a,b}y_{i,n}$, or both, would be a monomial in $B$ when $i \neq m $, $j \neq n$ (in the cases $i=m$ or $j=n$, $x_{a,b}$ can also produce monomials in $C$ and $D$). So, no $x$-vertex can be attached to $F$. Recall also that $F_y$ is a subset of a pair non-intersecting paths from $y_{1,1}$, $y_{2,1}$  to $y_{m,n}$, $y_{m,n-1}$ (Lemma \ref{lemma_Fx and Fy subsets of classical case}). So, if we attach a vertex $y_{c,d}$ to $F_y$, then $y_{c,d}$ must be in one of those two non-intersecting paths. If $i \neq m$, $j \neq n$ (see Figure \ref{threecases}), then $y_{c,d}$ must be  in $R_4$, but then $x_{i,j}y_{c,d}$ would be in $B$. If $i = m$, then $y_{c,d}$ must be  in $R_1$ and in row $m$. But then $x_{i,j}y_{c,d}$ and some $y$-variable that is in the upper path of $F_y$ and in $R_1$ would produce a monomial in $C$. Finally, if $j = n$, then $y_{c,d}$ must be  in $R_3$ and in column $n$. But then $x_{i,j}y_{c,d}$ and some $y$-variable that is in the lower path of $F_y$ and in $R_3$ would produce a monomial in $D$. So, no $y$-vertex can be attached to $F$ either. Thus, $F$ is maximal.

\begin{figure}[h]
\includegraphics[scale=0.30]{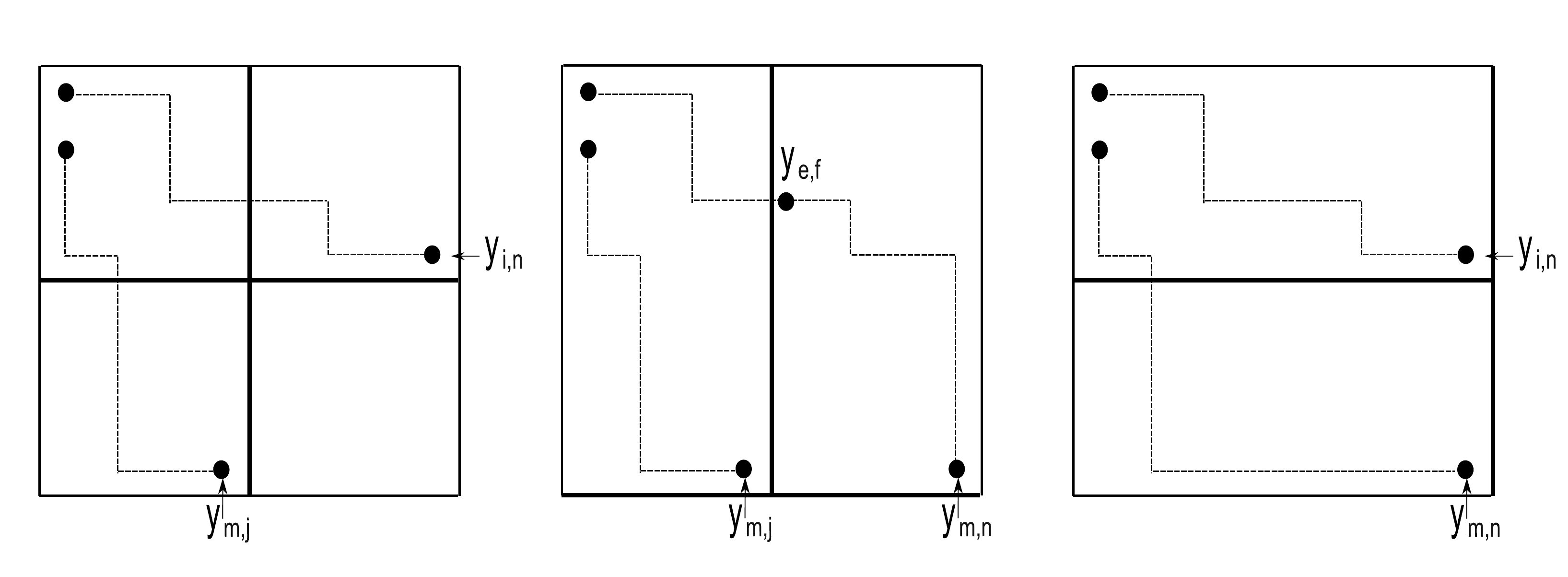}
\caption{ \label{threecases}}
\end{figure}


Finally, we will show that any facet $f=f_{x}f_{y}$ of $\SC$ with $\mu(f) = x_{i,j}$ must be of the form described in the theorem. Since $f_x$ is a subset of  a path from $x_{1,1}$ to $x_{m,n}$ (by Lemma \ref{lemma_Fx and Fy subsets of classical case}), and since $\mu(f) = x_{i,j}$, then it follows that $f_x$ must actually be a subset of a path from $x_{i,j}$ to $x_{m,n}$.

Next, again by Lemma \ref{lemma_Fx and Fy subsets of classical case}, $f_y$ must be a subset of a pair of non-intersecting paths from $y_{1,1}$, $y_{2,1}$  to $y_{m,n}$, $y_{m,n-1}$. By Lemma \ref{y's present}, it follows that $f_y$ must also contain $y_{i,n}$ and $y_{m,j}$.

Now, if $i \neq m$, $j \neq n$ (see Figure  \ref{threecases}), $f_y$ can not contain $y$-vertices in $R_4$, because $x_{i,j}$ and any vertex in that region is a monomial in $B$. So, $f_y$ must be a subset of a family of two non-intersecting paths from $y_{1,1}$, $y_{2,1}$ to $y_{i,n}$, $y_{m,j}$. Next, suppose $i = m$ (see Figure \ref{threecases}). Since $f_y$ is a subset of a pair of non-intersecting paths from $y_{1,1}$, $y_{2,1}$  to $y_{m,n}$, $y_{m,n-1}$, it is straightforward to verify, using maximality of $f$, that $f_y$ must contain the $y$-variable of the upper $y$-path that is furthest to north-west in $R_1$, call it $y_{e,f}$. Also, notice that there should be no $y$-vertices in $f_y \cap R_1$ such that one is to the south-west of the other (otherwise $x_{m,j}$ and those two $y$-vertices would produce a monomial in $C$). Therefore, $f_y \cap R_1$ must actually be a subset of a facet in $\Delta_2$ on vertex set $R_1$, i.e. $f_y \cap R_1$ must be a subset of some path in $R_1$ starting in
  $y_{e,f}$ and ending at $y_{m,n}$ (recall Proposition \ref{Prop_classical_facets}). So, $f_y$ must be a subset of a family of two non-intersecting paths from $y_{1,1}$, $y_{2,1}$ to $y_{i,n}$, $y_{m,j}$, $i=m$. Finally, we conclude the same result for the case $j=n$ (see Figure \ref{threecases}) using similar arguments from the case $i=m$.

Finally, notice that $f=f_xf_y$  is actually a subset of a some facet $F$ as described in the theorem. Maximality of $f$ implies that it actually has to be one of those facets $F$.

\end{proof}

Knowing the structure of a facet $F = F_xF_y$ of the simplicial complex $\SC$, we can easily count the number of vertices that $F$ is composed of, so we can determine $\dim F = |F|-1$. In particular, we see that the dimension of any facet $F$ is $2(m+n)-3$. Notice that the dimension of $F$ depends only on the constants $m$ and $n$. Thus, we can conclude that all facets of the simplicial complex $\SC$ have the same dimension, i.e. $\SC$ is a pure simplicial complex of dimension $2(m+n)-3$.

\begin{Corollary}
The dimension of $R/J$ is $2(m+n)-2$.
\end{Corollary}

Theorem \ref{facets} also allows to determine the total number of facets in $\SC$. Thus, we can determine the multiplicity of $R/J$ as well.

\begin{Corollary}
The multiplicity  of $R/J$ is given by

\begin{eqnarray}
 \sum \limits_{(i,j), (i,j) \neq (m,n)}   \binom{m+n-i-j}{m-i}     det \left(
\begin{array}{cc}
  \binom{i+n-2}{i-1}   &  \binom{m+j-2}{m-1} \\
  \binom{i+n-3}{i-2}   &  \binom{m+j-3}{m-2} \\
\end{array}%
\right)
\end{eqnarray}

\end{Corollary}

\begin{proof}
The number of paths from $x_{i,j}$ to $x_{m,n}$ is $\binom{m+n-i-j}{m-i}$, while the number of non-intersecting paths from $y_{1,1}$, $y_{2,1}$  to $y_{i,n}$, $y_{m,j}$ is given by (see \cite[\S 2.2]{Krat2})
\begin{eqnarray}
det \left(
\begin{array}{cc}
  \binom{i+n-2}{i-1}   &  \binom{m+j-2}{m-1} \\
  \binom{i+n-3}{i-2}   &  \binom{m+j-3}{m-2} \\
\end{array}%
\right) \nonumber
\end{eqnarray}

\end{proof}


\begin{Remark}
 Professor Sudhir Ghorpade (\cite{Ghor}) has shown that the expression for the multiplicity of $R/J$ above simplifies to ${{n+m-2}
\choose{m-1}}^2$ .
\end{Remark}

\section{Shellability of $\SC$} \label{SecnShellability}

The main goal of this section is to prove that our simplicial complex $\SC$ is shellable. Recall the following definition of shellability:

 \begin{Definition} \label{shellability_def}
 A simplicial complex $\Delta$ is \emph{shellable} if it is pure and if its facets can be given a total order, say $F_1$,$F_2$, $\dots$, $F_e$, so that the following condition holds: for all $i$ and $j$ with $1 \leq j < i \leq e$ there exists $v \in F_i \setminus F_j$ and an index $k$, $1 \leq k < i$, such that $F_i \setminus F_k = \{v\}$. A total order of the facets satisfying this condition is called \emph{shelling} of $\Delta$.
 \end{Definition}

 \begin{Theorem} \label{shellability_thm}
The simplicial complex $\SC$ is shellable.
 \end{Theorem}

 \begin{proof}
 Note that at the end of the previous section we have argued that $\SC$ is pure. We will proceed by first giving a partial order to the facets of $\SC$. Let $P=P_xP_y$ and $Q=Q_xQ_y$ be two facets of $\SC$. If $\mu(P)$ is in a row below $\mu(Q)$, we set $P<Q$ (see Figure \ref{partialorder}). If $\mu(P)$ and $\mu(Q)$ are in the same row, but $\mu(P)$ is to the right of $\mu(Q)$, we set $P<Q$ (see Figure \ref{partialorder}). If $\mu(P)=\mu(Q)$ but $P_x$ is to the right of $Q_x$ as one goes from $\mu(P)$ to $x_{m,n}$, then $P<Q$ (see Figure \ref{partialorder}). If $P_x = Q_x$ and the upper $y$-path of $P_y$ goes to the right of the upper $y$-path of $Q_y$, we set $P<Q$. Finally, if $P_x = Q_x$, the upper $y$-path of $P_y$ is the same as the upper $y$-path of $Q_y$ and the lower $y$-path of $P_y$ goes to the right of the lower $y$-path of $Q_y$, we set $P<Q$. Now we arbitrarily extend this partial order on the facets of $\SC$ to a total order.

\begin{figure}[h]
\includegraphics[scale=0.30]{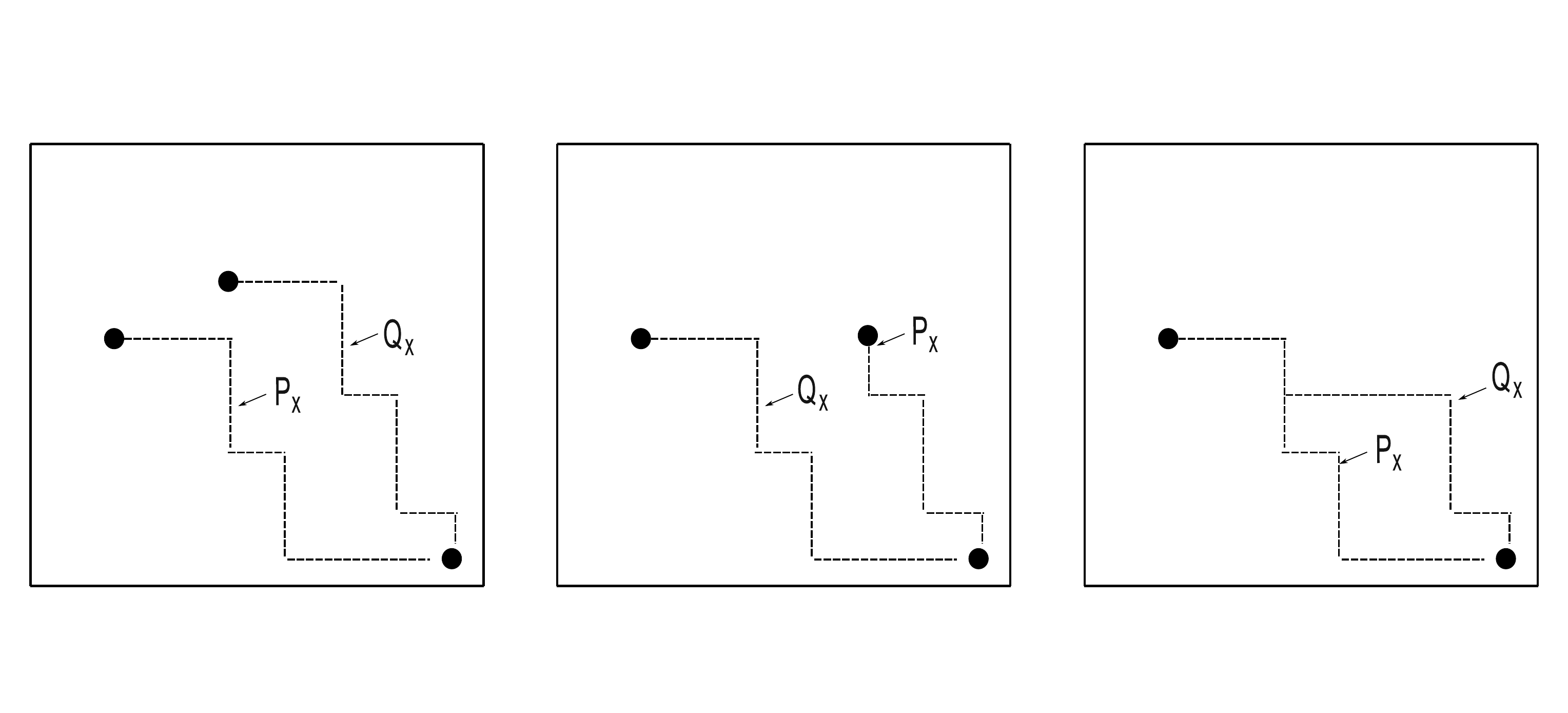}
\caption{ \label{partialorder}}
\end{figure}

 Now we will prove that the selected total order is indeed a shelling of $\SC$. Let $P=P_xP_y$ and $Q=Q_xQ_y$ be two facets of $\SC$ such that $P<Q$. Our goal is to find $v \in Q \setminus P$ and a facet $R < Q$ such that $Q \setminus R = \{v\}$.   Suppose that $\mu(P) \neq \mu(Q)$. Notice $P$ can not contain $\mu(Q) = x_{i,j}$ (otherwise $P<Q$ is contradicted). Take $v = x_{i,j}$. Take $R=R_xR_y$ to be the following: $R_x =  Q_x \setminus x_{i,j}$ and $R_y = Q_yy_{m,j+1}$ if $\mu(R) = x_{i,j+1}$ or $R_y = Q_yy_{i+1,n}$ if $\mu(R) = x_{i+1,j}$. In the special case $\mu(Q) = x_{m-1,n}$,  take $R_x = x_{m,n-1}x_{m,n}$, $R_y = Q_y$.

Next, suppose that $\mu(P) = \mu(Q)$, but $P_x \neq Q_x$. Then, there must be a right turn $H = x_{a,b}$ in $Q_x$ that is not in $P_x$ or else $Q_x$ would be to the right of $P_x$, contradicting $P<Q$. So, in this case take $v = H=x_{a,b}$ and $R=R_xR_y$ where $R_x = Q_x$ with $H=x_{a,b}$ replaced by $x_{a+1,b-1}$ and $R_y=Q_y$.

Next, suppose that $P_x = Q_x$ and the upper $y$-paths of the two facets are different. Notice that the upper path of $Q_y$ can not be strictly on the right of the upper path of $P_y$ (otherwise $P<Q$ is contradicted). So, there must be a right turn $H = y_{c,d}$ of the upper path of $Q_y$ strictly on the left of the upper path of $P_y$. Thus, $H = y_{c,d}$ can not be in $P_y$. So, take $v = y_{c,d}$. If $y_{c+1,d-1}$ is not in the lower path of $Q_y$, let $R=R_xR_y$ be the following facet: $R_x = Q_x$ and $R_y = Q_y$ with $y_{c,d}$ replaced by $y_{c+1,d-1}$. If $y_{c+1,d-1}$ is in the lower path of $Q_y$ (see Figure \ref{doublecascade}), then notice that $y_{c+1,d-1}$ must be a right turn as well. Then take $R=R_xR_y$ to be the following facet: $R_x = Q_x$ and $R_y$ is obtained from $Q_y$ be removing $y_{c,d}$ and by adding $y_{c+2,d-2}$.

\begin{figure}[h]
\includegraphics[scale=0.20]{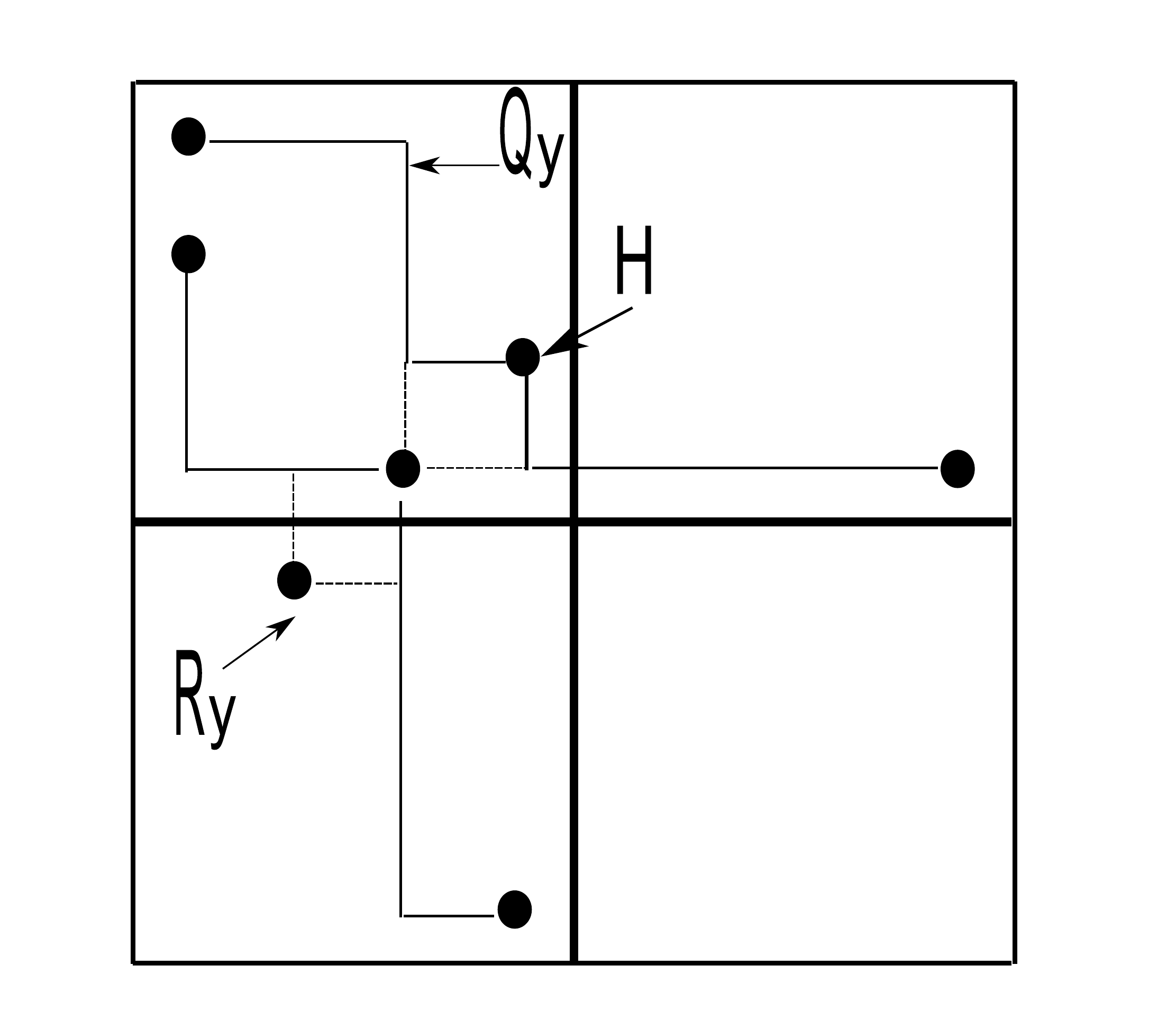}
\caption{ \label{doublecascade}}
\end{figure}

Finally, suppose that $P_x = Q_x$, the upper $y$-paths of the two facets are the same, but the lower $y$-paths are different. Similarly as in the previous paragraph, we see that there must be a right turn $H = y_{e,f}$ of the lower path of $Q_y$ strictly on the left of the lower path of $P_y$. Notice that $H = y_{e,f}$ can not be in the upper path of $P_y$ because it is the same as the upper path of $Q_y$. So, take $v = y_{e,f}$. Let $R=R_xR_y$ be the facet: $R_x = Q_x$ and $R_y = Q_y$ with $y_{e,f}$ replaced by $y_{e+1,f-1}$.

 \end{proof}

We are now in position to prove Theorem \ref{Cohen-Macaulay}, the main result of the paper:

%
%

\begin{proof}[Proof of Theorem \ref{Cohen-Macaulay}]
By standard
results, the ring $R/J$ is Cohen-Macaulay if the ring $R/\LT$ is
Cohen-Macaulay (see \cite[Corollary 8.31]{MiSt}.  By construction,
$R/\LT$ is precisely the Stanley-Reisner ring associated to $\SC$,
and since $\SC$ is shellable, $R/\LT$ will necessarily be
Cohen-Macaulay (see \cite[Theorem 5.1.13]{BrHe}).
\end{proof}

\end{document}